\documentclass[11pt,reqno]{amsart}

\usepackage{amsmath}
\usepackage{amssymb}
\usepackage{amsfonts}
\usepackage{hyperref}
\usepackage{microtype}
\usepackage{fullpage}
\usepackage{mathrsfs} 

\usepackage{tikz} \usetikzlibrary{calc} \tikzset{>=latex}

\DisableLigatures{encoding = *, family = * }

\newtheorem{theorem}{Theorem}

\newtheorem{corollary}{Corollary}
\newtheorem{lemma}{Lemma}
\newtheorem{remark}{Remark}

\numberwithin{equation}{section}
\numberwithin{equation}{section}
\numberwithin{theorem}{section}
\numberwithin{remark}{section}
\numberwithin{lemma}{section}
\numberwithin{figure}{section}

\usepackage{color} 
\allowdisplaybreaks
\usepackage{flushend}
\usepackage{verbatim} 

\title{Robust stabilization of hyperbolic PDE-ODE systems via Neural Operator-approximated gain kernels}
\subjclass[2020]{}

\author{Kaijing Lyu\textsuperscript{\,$\dagger$}}
\address{\textsuperscript{$\dagger$}\, School of Mathematics and Statistics,  Beijing Institute of Technology, 100081 Beijing, China.
	\newline \indent \textsuperscript{$\ast$}\, Chair of Computational Mathematics, DeustoTech, University of Deusto, Avenida de las Universidades 24, 48007 Bilbao, Basque Country, Spain
} 
\email{kjlv@bit.edu.cn}

\thanks{This project has received funding from the European Research Council (ERC) under the European Union's Horizon 2030 research and innovation programme (grant agreement NO: 101096251-CoDeFeL). UB was partially supported by the Basque Government through the consolidated research group "CMRG-Computational Mathematics Research Group" under Grant IT1856-26, by the Grant PID2023-146872OB-I00-DyCMaMod of MICIU (Spain), and by the COST Actions CA24122-Multiscale Stochastics, Patterns, and Analysis of Combinatorial Environments and CA24136-Interactions between Control Theory and Machine Learning.}

\author{Umberto Biccari\textsuperscript{\,$\ast$}} 
\email{umberto.biccari@deusto.es}

\author{Jun-Min Wang\textsuperscript{\,$\dagger$}}  
\email{jmwang@bit.edu.cn}


\begin{document}
	
\begin{abstract}
This paper investigates the mean square exponential stabilization problem for a class of coupled PDE-ODE systems with Markov jump parameters. The considered system consists of multiple coupled hyperbolic PDEs and a finite-dimensional ODE, where all system parameters evolve according to a homogeneous continuous-time Markov process. The control design is based on a backstepping approach. To address the computational complexity of solving kernel equations, a DeepONet framework is proposed to learn the mapping from system parameters to the backstepping kernels.  By employing Lyapunov-based analysis, we further prove that the controller obtained from the neural operator(NO) ensures stability of the closed-loop stochastic system. Numerical simulations demonstrate that the proposed approach achieves more than two orders of magnitude speedup compared to traditional numerical solvers, while maintaining high accuracy and ensuring robust closed-loop stability under stochastic switching.
\end{abstract}
	
\maketitle

\section{Introduction}
This paper studies mean-square exponential stabilization of coupled stochastic hyperbolic PDE–ODE systems. We design computationally efficient boundary controllers for mode-dependent dynamics by combining backstepping design with NO-based kernel approximation.

Coupled PDE--ODE systems arise naturally in several engineering applications, including traffic flow control, where hyperbolic PDEs describe the distributed traffic dynamics and ODEs model lumped dynamics at critical locations, with boundary coupling accounting for their interaction \cite{yang2025adaptive,liard2020pde}. In practice, these systems are often affected by uncertainties stemming from stochastic fundamental diagrams, time-varying traffic demand, random incidents, or weather conditions \cite{wang2012stochastically,zhang2017stochastic,zhang2024mean,auriol2023mean}. Similar stochastic PDE--ODE structures also appear in fluid transport networks and energy systems. Backstepping provides a constructive framework for boundary stabilization of such systems \cite{krstic2008boundary,krstic2008backstepping,di2018stabilization}, although the resulting kernel equations become significantly more challenging to solve in the presence of uncertain or mode-dependent parameters.

To overcome this computational burden, we employ neural operators as surrogates for the backstepping kernels. Neural operators learn mappings between function spaces and enable rapid evaluation once trained \cite{lu2021learning,garciauniversal}. The learned operator is then embedded into the control loop, enabling real-time implementation under stochastic switching.

Recent studies have explored neural-operator-based acceleration of controller design for deterministic PDE systems \cite{bhan2023neural, LV2025112553, wang2025backstepping}. More recently, operator-learning techniques have also been employed for robust stabilization of Markov-jump hyperbolic PDEs under parameter uncertainty \cite{zhang2026operator}. The present work addresses a different setting, namely coupled PDE–ODE systems, where the finite-dimensional dynamics introduce additional difficulties and requires an adapted Lyapunov analysis.

The contributions of this paper are summarized as follows:

\begin{itemize}
\item We study mean-square exponential stabilization of coupled PDE–ODE systems with Markov-jump parameters under neural-operator approximations of the backstepping kernels. Although the Lyapunov analysis follows the classical robustness paradigm for backstepping systems, the coupled PDE–ODE structure requires a different target system, new kernel equations, and a consequent adaptation of the stability argument.

\item To reduce the computational cost of solving the kernel equations, we employ an NO-based approximation. 

\item Numerical results show over two orders of magnitude speedup compared to standard solvers, while maintaining accuracy and closed-loop stability.
\end{itemize}

The paper is organized as follows. Section~\ref{sec:problem_statement} introduces the problem. Section~\ref{sec:backstepping} presents the nominal backstepping design. Section~\ref{sec:NO} establishes the NO approximation of the kernel equations, and Section~\ref{sec:NO_stabilization} analyzes mean-square stability under approximation errors. Section~\ref{sec:simulations} provides numerical validation, and Section~\ref{sec:conclusions} concludes the paper. 

\textbf{Notation:} We denote \(L^{2}([0,1],\mathbb{R})\) the space of real-valued square-integrable functions with norm \(\|f\|^{2} = \|f\|_{L^{2}}^{2}\). The supremum norm is denoted by \(\|\cdot\|_{\infty}\). For a random signal \(x(t)\), we denote the conditional expectation of \(x(t)\) at the instant \(t\) with initial condition \(x_{0}\) at instant \(s \leq t\) as \(\mathbb{E}_{[s,x_{0}]}(x(t))\). 
	
\section{Problem formulation}\label{sec:problem_statement}
In this paper, we consider the following stochastic ODE-PDE system 
\begin{equation}\label{eq:5}
	\begin{cases}
	    \dot{X}(t) = AX(t)+Bz(0,t) & t \in\mathbb{R}^+
	    \\
	    \quad X(0) = X_0 
	    \\
	    \partial_t w(x,t) = -\Lambda^{+}(t)\partial_x w(x,t)+\Sigma^{++}(t)w(x,t) + \Sigma^{+-}(t)z(x,t) & (x, t) \in [0, 1] \times \mathbb{R}^+
	    \\
	    \partial_t z(x,t) = \Lambda^{-}(t)\partial_x z(x,t)+\Sigma^{-+}(t)w(x,t) + \Sigma^{--}(t)z(x,t) & (x, t) \in [0, 1] \times \mathbb{R}^+ 
	    \\
	    w(0,t) = Q(t)z(0,t)+CX(t) & t \in\mathbb{R}^+
	    \\
	    z(1,t) = R(t)w(0,t)+U(t) & t \in\mathbb{R}^+
	    \\
	    w(x,0) = w_0, \; z(x,0) = z_0 & x\in [0, 1]
	\end{cases}
\end{equation}
where $X\in\mathbb{R}^2$ is the ODE state, $\mathbf{w} =(w,z)\in\mathbb{R}^3\times\mathbb{R}$ are the PDEs states, $A\in\mathbb{R}^{2\times2}, B\in \mathbb{R}^2$ and $C\in \mathbb{R}^{3\times2}$ are given matrices, the function $U(t):\mathbb{R}^+\to\mathbb{R}$ is a control input to be designed for stabilizing the dynamics, and
\begin{equation}\label{eq:P}
	\mathscr{S} = \Big\{\Lambda^{+}, \Lambda^{-}, \Sigma^{++},\Sigma^{+-}, \Sigma^{-+},\Sigma^{--}, Q, R\Big\} 
\end{equation}
where
\begin{displaymath}
	\begin{array}{llll}
		\Lambda^{+}\in\mathbb{R}^{3\times 3}, & \Lambda^{-}\in \mathbb{R}, & \Sigma^{++}\in \mathbb{R}^{3\times 3}, & \Sigma^{+-}\in \mathbb{R}^{3\times 1}, 
		\\
		\Sigma^{-+}\in \mathbb{R}^{1\times 3}, & \Sigma^{--}\in \mathbb{R}, & Q\in \mathbb{R}^{3\times 1}, & R\in \mathbb{R}^{1\times 3}.	
	\end{array} 
\end{displaymath}

The matrices $A$, $B$, and $C$ are assumed deterministic, while only the transport velocities and coupling coefficients evolve according to the Markov process. This allows us to focus on the effect of stochasticity in the distributed dynamics. Extending the analysis to stochastic matrices $A$, $B$, and $C$ is left for future work.

$\mathscr{S}$ is a set of stochastic parameters following continuous Markov processes.  Each random element $S$ of the set $\mathscr{S}$ is a Markov process with the following properties:
\begin{align*}
    S(t)\in\Big\{S_1,S_2, \cdots, S_{r_S}\Big\},    
\end{align*}
whose realization is right continuous. We assume that there exists known lower and upper bounds $\underline{S}< \overline{S}<+\infty$ such that for all $j$, $\underline{S} \leq S_{j} \leq \overline{S}$. Moreover, we will consider the matrix $\Lambda^+(t)\in\mathbb{R}^{3\times 3}$ to be diagonal, that is, 
\begin{align*}
	\Lambda^{+}(t)=\text{diag}\Big(\lambda_{1}(t), \lambda_{2}(t), \lambda_{3}(t)\Big),
\end{align*}
and the lower bounds of $\Lambda^\pm$ are positive: $\underline{\Lambda_j^+}$ and $\underline{\Lambda_j^-}>0$. 
	
Finally, for all $i,j \in\{1, \ldots, r\}$ and $0 \leq t_1 \leq t_2$, we shall denote $P_{i j}(t_1, t_2)$ the probability to switch from mode $S_{i}$ at time $t_1$ to mode $S_j$ at time $t_2$. They satisfy 
\begin{align*}
	P_{i j}: \mathbb{R}^2 \rightarrow[0,1] \quad\text{ with }\quad \sum_{j=1}^r P_{i j}(t_1, t_2)=1,   
\end{align*}
and follow the Kolmogorov equation \cite{hoyland2009system,kolmanovsky2001mean} 
\begin{equation}\label{eq:kolmogorov}
    \begin{array}{l}
        \displaystyle\partial_{t} P_{i j}(\varrho, t)=-c_j(t) P_{i j}(\varrho, t)+\sum_{k=1}^r P_{i k}(\varrho, t) \tau_{k j}(t), 
        \\[15pt]
        \displaystyle P_{i i}(\varrho, \varrho)=1 \quad\text{ and }\quad P_{i j}(\varrho, \varrho)=0 \text { for } i \neq j,
    \end{array}
\end{equation}
where $\tau_{i j}(t):\mathbb{R}^+\to \mathbb{R}^+$ such that 
\begin{align*}
	\tau_{i i}(t)=0 \quad\text{ and }\quad \tau_{i j}(t)\leq \tau^{\star}	c_j(t)=\sum_{k\neq j=1}^r \tau_{j k}(t):\mathbb{R}^+\to \mathbb{R}^+    
\end{align*}
are non-negative-valued functions. We define the state vector $\alpha(t)$ as a set including all Markov-jumping parameters at time $t$, as follows:
\begin{align}
	\alpha(t) = \Big\{\Lambda^{+}(t), \Lambda^{-}(t), \Sigma^{++}(t), \quad\Sigma^{+-}(t), \Sigma^{-+}(t),\Sigma^{--}(t), Q(t), R(t)\Big\}.
\end{align}

Let $\mathfrak{R}$ denote the Cartesian product of the index sets $\{1, \dots, r_S\}$ for all $S \in \mathscr{S}$ with a finite number of states $r=r_{\Lambda^{+}} \times r_{\Lambda^{-}} \times r_{\Sigma^{++}} \times r_{\Sigma^{+-}} \times r_{\Sigma^{-+}} \times r_{\Sigma^{--}} \times r_{R} \times r_{Q}$. Each element $j \in \mathfrak{R}$ is the indices of each random parameter. We say that $\alpha(t) = \alpha_j$ if $S(t) = S_{j_S}$ for all $S \in \mathscr{S}$. 

Our main objective is to efficiently design $U$ so to guarantee the mean-square closed-loop stability for \eqref{eq:5}. To do so, we will first consider the system in its nominal version, where the stochastic coefficients $\alpha(t)$ are replaced by constant ones
\begin{displaymath}
    \alpha_0= \Big\{\Lambda^{+}_0, \Lambda^{-}_0, \Sigma^{++}_0, \Sigma^{+-}_0, \Sigma^{-+}_0,\Sigma^{--}_0, Q_0, R_0\Big\}, 
\end{displaymath}
and we will build a backstepping controller stabilizing the dynamics. We will later show that this same controller is capable of stabilizing also \eqref{eq:5}, provided the nominal parameters are sufficiently close to the stochastic ones on average. 
	
\section{Backstepping controller design}\label{sec:backstepping}
	
We consider here that the stochastic parameters are in the nominal mode \(\alpha(t)  = \alpha_0\), that is, we consider the system
\begin{equation}\label{eq:system_nomimal}
	\begin{cases}
	    \dot{X}(t) = A X(t) + B z_{nom}(0,t) & t\in\mathbb{R}^+
	    \\
	    \partial_t w_{nom}(x,t) = - \Lambda_0^+ \partial_x w_{nom}(x,t) + \Sigma_0^{++} w_{nom}(x,t) + \Sigma_0^{+-} z_{nom}(x,t) & (x, t) \in [0, 1] \times \mathbb{R}^+
	    \\
	    \partial_t z_{nom}(x,t) = \Lambda_0^- \partial_x z_{nom}(x,t) + \Sigma_0^{-+} w_{nom}(x,t) +\Sigma_0^{--} z_{nom}(x,t) & (x, t) \in [0, 1] \times \mathbb{R}^+
	    \\
	    w_{nom}(0,t) = Q_0 z_{nom}(0,t) + C X(t) & t\in\mathbb{R}^+
	    \\
	    z_{nom}(1,t) = R_0 w_{nom}(0,t) + {U}(t) & t\in\mathbb{R}^+.
	\end{cases}
\end{equation}
	
Since $(A,B)$ is stabilizable, there exists a gain $F \in \mathbb{R}^{1\times 2}$ such that $A+BF$ is Hurwitz. Fix such a gain $F$. We want to employ backstepping to design a control $U$ to stabilize $X$ in \eqref{eq:system_nomimal}, that is, $X(t)\to 0$ as $t\to +\infty$. To do that, we introduce the following Volterra transformation between the states $(X,w_{nom},z_{nom})$ and $(X,\theta,\rho)$
\begin{equation}\label{eq:transformation}
    \begin{aligned}
        \theta(x,t) =& w_{nom}(x,t),
        \\
        \rho(x,t) =& z_{nom}(x,t) - \int_0^x K_0(x, \xi) w_{nom}(\xi, t)\, d \xi +\int_0^x N_0(x, \xi) z_{nom}(\xi, t)\, d \xi - \gamma(x)X(t)
    \end{aligned}
\end{equation}
with ${\theta}=(\theta_1,\theta_2,\theta_3)$ and where 
\begin{align*}
	K_0(x, \xi)\in \mathbb{R}^{1 \times 3}, \quad N_0(x, \xi) \in \mathbb{R}\quad\text{ and }\quad \gamma(x)\in \mathbb{R}^{1\times 2}	
\end{align*}
are kernel functions to be determined on the triangular domain $\mathcal{T} = \{0 \leq \xi \leq x \leq 1\}$, so that \eqref{eq:system_nomimal} is mapped into the following target system
\begin{equation}\label{eq:target-1}
\begin{cases}
    \dot{X}(t) = (A+BF)X(t)+B\rho(0,t), & t\in\mathbb{R}^+
    \\
    \theta_t(x, t) = \displaystyle -\Lambda_0^+ \theta_x + \Sigma_0^{++} \theta+ \Sigma_0^{+-} \rho+ D_0(x)X  & (x, t) \in [0, 1] \times \mathbb{R}^+
    \\[5pt]
    \hphantom{\theta_t(x, t) = } \displaystyle +\int_0^x \Big(C_0^+(x, \xi) \theta +C_0^-(x, \xi) \rho\Big)\, d \xi 
    \\
    \rho_t(x, t) = \Lambda_0^- \rho_x(x, t) + {\Sigma}_0^{--} \rho(x, t) & (x, t) \in [0, 1] \times \mathbb{R}^+
    \\
    \theta(0, t) = Q_0 \rho(0, t)+C_0 X(t), & t\in\mathbb{R}^+
    \\
    \rho(1, t) = 0 & t \in\mathbb{R}^+
\end{cases}
\end{equation}

In \eqref{eq:target-1}, $C_0^+(x, \xi) \in \mathbb{R}^{3 \times 3}$, $C_0^-(x, \xi) \in \mathbb{R}^{3 \times 1}$, $D_0 \in \mathbb{R}^{3 \times 1}$ and $C_0 \in \mathbb{R}^{3 \times 1}$ are bounded coefficients defined on $\mathcal T$ as
\begin{displaymath}
    \begin{aligned}
        & C_0^+(x, y) = \Sigma^{+-}_0K_0(x, y) + \int_y^x C_0^-(x, s)K_0(s, y)\,ds 
        \\
        & C_0^-(x, y) = \Sigma^{+-}_0N_0(x, y) + \int_y^x C_0^-(x, s)N_0(s, y)\,ds  
        \\
        & D_0(x) = \Sigma^{+-}_0\gamma(x) + \int_0^x C_0^-(x, y)\gamma(y)\, dy  
        \\
        & C_0 = C + Q_0 F
    \end{aligned}
\end{displaymath}

Moreover, we shall define $\mathcal L_0: \mathbf{w}=(w,z) \mapsto \Theta = (\theta,\rho)$ the map associated with \eqref{eq:transformation}. By differentiating \eqref{eq:transformation} with respect to $x$ and $t$, we obtain that the kernels ${K}_0(x, \xi)$, $N_0(x, \xi) $ and $\gamma(x)$ satisfy the following equations over $\mathcal T$
\begin{equation}\label{eq:kernel_1}
    \begin{cases}
        \Lambda_0^-(K_0)_x(x, \xi) = (K_0)_{\xi}(x, \xi) \Lambda_0^+ +N_0(x, \xi)\Sigma_0^{-+} + K_0(x, \xi) (\Sigma_0^{++}-\Sigma_0^{--}\mathbf{I}_3)
        \\
        \Lambda_0^-(N_0)_x(x, \xi)+\Lambda_0^-(N_0)_{\xi}(x, \xi) = K_0(x, \xi) \Sigma_0^{+-}
        \\
        K_0(x, x)\left(\Lambda_0^- \mathbf{I}_3+\Lambda_0^+\right)  =-\Sigma_0^{-+} 
        \\
        \Lambda_0^- N_0(x, 0)-K_0(x, 0) \Lambda_0^+ Q_0 =\gamma(x)B
        \\
        \Lambda_0^-\gamma'(x) =\gamma(x)A-\Sigma_0^{--}\gamma(x)-K_0(x,0)\Lambda_0^+C
        \\
        \gamma(0)=F
	\end{cases}
\end{equation}
where $\mathbf{I}_3$ is the $3 \times 3 $ identity matrix. The condition $\gamma(0)=F$ guarantees that 
\begin{align*}
	\rho(0,t)=z_{nom}(0,t)-FX(t),	
\end{align*}
and therefore the left boundary condition becomes 
\begin{align*}
	\theta(0,t)=Q_0\rho(0,t)+(C+Q_0F)X(t),
\end{align*}
consistently with the target system. The derivations of kernel PDEs is given in Appendix \ref{app:kernel}. Moreover, since the backstepping transformation is invertible, with inverse in the same form
\begin{align}
    z_{nom}(x,t)= \rho(x,t) +\int_0^xL_0(x,\xi)w_{nom}(\xi,t)d\xi +\int_0^x M_0(x,\xi)\rho(\xi,t) d\xi +T_0(x)X(t) \label{eq:inverse_2}
\end{align}
where $L_0, M_0$ and $T_0$ are inverse transformation kernels, there exist positive constants $b_1,b_2>0$ such that
\begin{align}\label{backstepping_equivalence}
    b_1 \Big(\|w(t)\| + \|z(t)\| + |X|\Big)^2 \leq \Big(\|\theta(t)\| + \|\rho(t)\| + |X|\Big)^2 \leq b_2 \Big(\|w(t)\| + \|z(t)\|+ |X|\Big)^2.
\end{align}

For the prescribed stabilizing gain $F$, the kernel equations \eqref{eq:kernel_1} admit a unique solution, as guaranteed by the following Lemma. 
\begin{lemma}\label{Bound}
Let $\mathbb{D}_3(\mathbb{R})$ denote the space of all $3 \times 3$ real diagonal matrices, and consider a compact set
\begin{align*}
    \Omega \subset \mathbb{D}_3(\mathbb{R}) \times \mathbb{R}^2 \times \mathbb{R}^{3\times3} \times (\mathbb{R}^{3\times1})^2 \times (\mathbb{R}^{1\times3})^2.   
\end{align*}
Then, for all $\alpha_0\in\Omega$ there exists a unique solution
\begin{align*}
    (K_0,N_0,\gamma) \in L^\infty(\mathcal T;\mathbb R^{1\times3})\times L^\infty(\mathcal T;\mathbb R)\times L^\infty([0,1];\mathbb R^{1\times2})
\end{align*}
of \eqref{eq:kernel_1}, for the fixed gain $F$. Moreover the operator 
\begin{align*}
	& \mathcal{K}: \Omega \mapsto L^\infty(\mathcal T;\mathbb R^{1\times3})\times L^\infty(\mathcal T;\mathbb R)\times L^\infty([0,1];\mathbb R^{1\times2})
	\\
   	& \mathcal{K}(\alpha_0)(x,y) = \Big(K_0(x,\xi),N_0(x,\xi),\gamma(x)\Big)	
\end{align*}
is locally Lipschitz.
\end{lemma}

\begin{proof} 
Let us denote $\mathcal T_1=\mathcal T\times\mathcal T\times[0,1]$. We prove that there exists a constant $h_\Omega>0$, depending only on the compact set $\Omega$, such that for all $\alpha_a,\alpha_b\in \Omega$,
\[
\|K(\alpha_a)-K(\alpha_b)\|_{(L^\infty(\mathcal T_1))^3}\le h_\Omega\|\alpha_a-\alpha_b\|,
\]
where $K(\alpha)=(K_0(\alpha),N_0(\alpha),\gamma(\alpha))$ and $\|\cdot\|_{(L^\infty(\mathcal T_1))^3}$ denotes the product norm.

For each fixed $\alpha_0\in\Omega$, the characteristic construction for the kernel equations gives a unique bounded solution $K(\alpha_0)\in(L^\infty(\mathcal T_1))^3$; see [3, Theorem 4.1]. Since $\Omega$ is compact and all transport velocities are uniformly bounded away from zero on $\Omega$, the characteristic maps and all coefficients in the corresponding integral equations are uniformly bounded. Hence there exists $M_\Omega>0$ such that
\[
\sup_{\alpha_0\in\Omega}\|K(\alpha_0)\|_{(L^\infty(\mathcal T_1))^3}\le M_\Omega.
\]

Denote $K_a:=K(\alpha_a)$ and $K_b:=K(\alpha_b)$, and let $\Delta K:=K_a-K_b$. Subtract the characteristic integral equations associated with $\alpha_a$ and $\alpha_b$. The terms containing only the coefficients or boundary data are bounded by $c_1\|\alpha_a-\alpha_b\|$, because these quantities depend smoothly on the parameters on the compact set $\Omega$. The remaining terms contain $\Delta K$ under Volterra integrals. Consequently, for
\[
H(x):=\sup_{0\leq \xi\leq s\leq x}\big|\Delta K(s,\xi)\big|,
\]
there exist constants $h_1,h_2>0$, depending only on $\Omega$, such that
\[
H(x)\le h_1\|\alpha_a-\alpha_b\|+h_2\int_0^xH(s)\,ds,\qquad 0\le x\le1.
\]
The Volterra-Gr\"onwall inequality therefore gives
\[
\|\Delta K\|_{(L^\infty(\mathcal T_1))^3}=H(1)\le h_1e^{h_2}\|\alpha_a-\alpha_b\|.
\]
Thus the asserted Lipschitz estimate holds with $h_\Omega=h_1e^{h_2}$. This completes the proof.
\end{proof}

\medskip 
\noindent Finally, using the boundary identity 
\begin{align*}
	\rho(1,t)=\, z_{nom}(1,t)-\int_0^1 K_0(1,\xi)w_{nom}(\xi,t)\,d\xi +\int_0^1 N_0(1,\xi)z_{nom}(\xi,t)\,d\xi-\gamma(1)X(t)
\end{align*} 
and imposing $\rho(1,t)=0$, we can design a stabilizing control law for \eqref{eq:system_nomimal} as follows 
\begin{align}
    U(t) = -R_0w_{nom}(0,t) + \int_0^1 K_0(1,\xi)w_{nom}(\xi,t)\,d\xi - \int_0^1 N_0(1,\xi)z_{nom}(\xi,t)\,d\xi + \gamma(1)X(t). \label{eq:U} 
\end{align}

\begin{theorem}\label{thm:4.2}
Assume that $(A, B)$ are stabilizable, and let $F$ be chosen so that $A+BF$ is Hurwitz. Define the control law $U$ as in \eqref{eq:U}, where $K_0$, $N_0$ and $\gamma$ are given by \eqref{eq:kernel_1}. If $C^+_0, C^-_0 \in \mathcal{L}^{\infty}(\mathcal{T})$, then \eqref{eq:system_nomimal} admits a zero equilibrium which is exponentially stable in the $L^2$ sense.
\end{theorem}

\begin{proof}
Consider the Lyapunov functional
\[
V_0(t)=\int_0^1 (L_0w(x,t))^T \mathcal D_0(x)L_0w(x,t)\,dx +X(t)^TPX(t),
\]
where
\[
\mathcal D_0(x)=\operatorname{Diag}\left\{\frac{e^{-\nu x/\lambda_1^0}}{\lambda_1^0}, \frac{e^{-\nu x/\lambda_2^0}}{\lambda_2^0}, \frac{e^{-\nu x/\lambda_3^0}}{\lambda_3^0}, \frac{ae^{\nu x/\Lambda_0^-}}{\Lambda_0^-}\right\},
\]
with \(a,\nu>0\), and where \(P=P^T>0\) solves
\[
P(A+BF)+(A+BF)^TP=-Q,
\]
for some \(Q=Q^T>0\). Since the backstepping transformation is boundedly invertible, \(V_0\) is equivalent to
\[
\|w(t)\|^2+\|z(t)\|^2+|X(t)|^2.
\]
For clarity, set
\[
E_\theta:=\sum_{m=1}^3\int_0^1\frac{e^{-\nu x/\lambda_m^0}}{\lambda_m^0}\theta_m^2(x,t)\,dx,
\qquad
E_\rho:=a\int_0^1\frac{e^{\nu x/\Lambda_0^-}}{\Lambda_0^-}\rho^2(x,t)\,dx.
\]

Taking the time derivative of \(V_0\), substituting the target dynamics (III.3), and integrating the transport terms by parts, the transport contributions are
\[
-\sum_{m=1}^3e^{-\nu/\lambda_m^0}\theta_m^2(1,t)+|\theta(0,t)|^2-\nu E_\theta
\]
and
\[
-a\rho^2(0,t)-\nu E_\rho,
\]
respectively, because \(\rho(1,t)=0\). The remaining boundary relation is
\[
\theta(0,t)=Q_0\rho(0,t)+C_0X(t),\qquad \rho(1,t)=0.
\]

We now make the choice of the Lyapunov parameters in an order which avoids any circular dependence. Fix a scalar \(q>4\|C_0\|^2+2\) and set \(Q=qI\). The Hurwitz property of \(A+BF\) then gives a unique \(P=P^T>0\). The boundary terms satisfy
\[
|\theta(0,t)|^2\le 2\|Q_0\|^2\rho^2(0,t)+2\|C_0\|^2|X(t)|^2
\]
and, by Young's inequality in the \(Q\)-metric,
\[
2X^TPB\rho(0,t)\le \frac{q}{2}|X(t)|^2+\frac{2}{q}\|PB\|^2\rho^2(0,t).
\]
Having fixed \(q\) and hence \(P\), choose
\[
a>2\|Q_0\|^2+\frac{2}{q}\|PB\|^2+1.
\]

Thus the negative boundary contribution \(-a\rho^2(0,t)\) absorbs all the positive \(\rho^2(0,t)\)-terms. Notice that this is a genuine sequential choice: the quantity \(\|PB\|\) is known before \(a\) is selected.

It remains to treat the bounded in-domain coupling terms. With \(q\), \(P\), and \(a\) fixed, the standard weighted transport estimate (Cauchy-Schwarz and Young's inequality applied with the weights \(\mathcal D_0\)) gives a constant \(c_1>0\), independent of \(\nu\), such that, for every sufficiently large \(\nu\), the terms containing only the distributed states are bounded above by
\[
\frac{\nu}{2}\big(E_\theta+E_\rho\big),
\]
while the terms coupling \(X\) to the distributed states are bounded by \(c_1\nu^{-1}|X(t)|^2\). This is the usual weighted-energy estimate for a heterodirectional hyperbolic system; the exponential weights make the bulk damping \(-\nu(E_\theta+E_\rho)\) dominate the bounded in-domain coefficients. We therefore choose \(\nu\) sufficiently large that \(c_1\nu^{-1}<q/2-2\|C_0\|^2\). Combining this estimate with the Lyapunov equation, the transport identities above, and the boundary estimates yields positive constants \(\kappa_X,\kappa_\rho,\kappa_\theta,\kappa_b>0\) for which
\[
\dot V_0(t)\le -\kappa_X|X(t)|^2-\kappa_\rho E_\rho-\kappa_\theta E_\theta
-\kappa_b\rho^2(0,t)-\sum_{m=1}^3e^{-\nu/\lambda_m^0}\theta_m^2(1,t).
\]
Since \(V_0\) is equivalent to \(E_\theta+E_\rho+|X|^2\), this estimate yields a constant \(\eta_0>0\) such that
\[
\dot V_0(t)\le -\eta_0 V_0(t).
\]

The computations follow the standard Lyapunov estimate for backstepping-stabilized coupled hyperbolic PDE-ODE systems; see, e.g., \cite{di2018stabilization}. Therefore,
\[
V_0(t)\le V_0(0)e^{-\eta_0 t},
\]
and, by the equivalence between \(V_0\) and the \(L^2\)-norm of the original variables, the zero equilibrium of (III.1) is exponentially stable in the \(L^2\) sense.	
\end{proof}
	
\section{DeepONet approximation of the backstepping kernels}\label{sec:NO}
	
To implement the controller \eqref{eq:U}, it is necessary to compute the kernels $K_0$, $N_0$ and $\gamma$ by solving \eqref{eq:kernel_1}. This, however, is most often computationally expensive, especially for systems with high spatial resolution. This computational burden limits the applicability of backstepping-based controllers in real-time scenarios.
	
To overcome this challenge, we leverage the DeepONet to learn the mapping from system parameters to the kernel functions. DeepONet is a neural network architecture introduced in \cite{lu2021learning} for learning operators. It is by now well-known (see \cite[Theorem 2.1]{Approximation}) that DeepONet allows to approximate continuous operators defined over compact sets up to any desired tolerance $\varepsilon$ depending only on the network structure. 

In particular, in our case, we have the following result as a direct consequence of \cite[Theorem 2.1]{Approximation} and Lemma \ref{Bound}. 

\medskip 
\begin{corollary}\label{th:NO}
For all $\varepsilon> 0$, there exists a neural operator 
\begin{align*}
	\hat{\mathcal K}: \Omega \to L^{\infty}(\mathcal T;\mathbb R^{1\times 3})\times L^{\infty}(\mathcal T;\mathbb R)\times L^{\infty}(0,1;\mathbb R^{1\times2})	
\end{align*}
such that, for all $\alpha_0\in\Omega$,
\begin{align*}
    & \|\hat K_0(\alpha_0)-K_0(\alpha_0)\|_{L^\infty(\mathcal T)} \leq\varepsilon,
    \\
    & \|\hat N_0(\alpha_0)-N_0(\alpha_0)\|_{L^\infty(\mathcal T)} \leq\varepsilon,
    \\
    & \|\hat \gamma(\alpha_0)-\gamma(\alpha_0)\|_{L^\infty(0,1)} \leq\varepsilon.
\end{align*}
\end{corollary}

\medskip 
As we shall see, once trained, DeepONet will enable rapid prediction of $K_0$, $N_0$, and $\gamma$, significantly reducing the computation time. With the learned neural operator $\hat{\mathcal{K}}$, we can define the following NO-approximated nominal control law
\begin{align}\label{eq:U_NO}
    \hat{U}=-R_0{w}(0, t) + \int_0^1\hat{K}_0(1,\xi)w(\xi,t)\,d \xi -\int_0^1\hat{N}_0(1,\xi)z(\xi,t)\,d \xi +\hat{\gamma}(1)X(t) 
\end{align}
Moreover, the effectiveness of \eqref{eq:U_NO} as a controller will be demonstrated in the next Section \ref{sec:NO_stabilization}. 

\medskip 
\begin{remark}
\em{The backstepping operator $\mathcal{L}_0$ and its inverse are constructed for the nominal constant-coefficient system and retained under Markov switching. Mean-square exponential stability of the closed loop is established provided the stochastic parameters remain sufficiently close to their nominal values on average.
}
\end{remark}

\section{Lyapunov analysis for the stochastic system under the NO-approximated control law}\label{sec:NO_stabilization}
	
In this section, we prove that the NO-approximated control law \eqref{eq:U_NO} can stabilize the stochastic system \eqref{eq:5}, provided the nominal parameters $S_0$ are sufficiently close to the stochastic ones on average. More precisely, we show the following sufficient condition for robust stabilization.
	
\begin{theorem}\label{them_Main2}
	There exist constants $\delta^* > 0$ and $\varepsilon^* > 0$ such that for all $\varepsilon \in (0,\varepsilon^*)$, if
	\begin{align*}
		\sum_{S\in\mathscr{S}} \mathbb{E}_{[0,S(0)]} \big( | S(t) - S_0 | \big) \leq \delta^*, \quad\text{ for all } t\in\mathbb{R}^+,
	\end{align*}
	the closed-loop system \eqref{eq:5} with control \eqref{eq:U_NO} is mean-square exponentially stable: there exist $\varsigma,\zeta > 0$ such that
	\begin{align*}
		\mathbb{E}_{(0, (p(0),S(0)))} \big[ p(t) \big] \leq \varsigma e^{-\zeta t} p(0),
	\end{align*}
	where $p(t) = \|\mathbf{w}(x,t)\|^2 + |X(t)|^2$ and $\mathbb{E}_{(0, (p(0),S(0)))}$ denotes the conditional expectation at time $t$ on the initial conditions $p(0)$ and $S(t) =S(0)$.
\end{theorem}

\medskip 
\begin{remark}
\em{Theorem~\ref{them_Main2} guarantees mean-square exponential convergence of the coupled PDE–ODE state despite two independent perturbation mechanisms: the stochastic variations of the system parameters and the approximation error introduced by the neural operator. The proof shows that both perturbations enter the Lyapunov estimate as additive terms, which can be absorbed into the nominal exponential decay whenever they remain sufficiently small.}
\end{remark}

\subsection{Target system in stochastic mode $\alpha_j$}\label{sec:target_stochastic}
Consider that $\alpha(t) = \alpha_j$ at time $t$. Define $\Psi = (\Theta, X)$, with $\Theta = (\theta, \rho) = \mathcal{L}_0 \textbf{w}$ the output of the transformation \eqref{eq:transformation} applied to states of the original stochastic systems \eqref{eq:5}. 
\begin{align}
    \dot{X}(t) =& (A+BF)X(t)+B\rho(0,t)\label{eq:target2-1}
    \\
    \theta_t(x, t) =& -\Lambda_j^+ \theta_x(x, t) + \Sigma_j^{++} \theta + \Sigma_j^{+-} \rho+D_j(x)X  +\int_0^x \Big(C_j^+(x, \xi) \theta+ C_j^-(x, \xi) \rho\Big) d \xi 
    \\
    \rho_t(x, t) =& \Lambda_j^- \rho_x(x, t) + \Sigma_j^{--} \rho(x, t) + f_{1j}(x)w(x,t) + f_{2j}(x)z(0,t)+f_{3j}(x)X(t)\nonumber
    \\
    &+\int_0^xf_{4j}(x,\xi)w\,d \xi +\int_0^xf_{5j}(x,\xi)z\,d \xi,\label{eq:target2-4}
    \\
    \theta(0, t) =&  Q_j \rho(0, t)+C_j X(t) 
    \\
    \rho(1, t) =& (R_j-R_0)\theta(0, t)+\Gamma(t)\label{eq:target2-5} 
\end{align}
where
\begin{equation*}
	\begin{aligned}
		C_j^+(x,y) & = \Sigma^{+-}_j K_0(x,y) + \int_y^x C_j^-(x,s)K_0(s,y)\,ds, 
		\\
		C_j^-(x,y) & = \Sigma^{+-}_j N_0(x,y) + \int_y^x C_j^-(x,s)N_0(s,y)\,ds, 
		\\
		D_j(x) & = \Sigma^{+-}_j \gamma(x) + \int_0^x C_j^-(x,y)\gamma(y)\,dy, 
		\\
		C_j &= C + Q_j F, 
		\\
		\Gamma(t) &= -\int_0^1 \Big(\hat{N}_0(1,\xi)-N_0(1,\xi)\Big) z(\xi,t)\,d\xi +\Big(\hat{\gamma}(1)-\gamma(1)\Big)X(t) 
		\\
		&\quad + \int_0^1\Big(\hat{K}_0(1,\xi)-K_0(1,\xi)\Big)w\,d\xi.
	\end{aligned}
\end{equation*}

The coefficients $f_{1j},\ldots,f_{5j}$ are the residual mismatch terms induced by applying the nominal backstepping map $\mathcal L_0$ to mode $\alpha_j$; they vanish when $\alpha_j=\alpha_0$ and satisfy the estimate of Lemma \ref{lemma4.1}. Moreover, since the inverse backstepping map \eqref{eq:inverse_2} is bounded, there exists a constant $c_{\mathcal L}>0$ such that
\begin{align}
	\|w(t)\|^2 + \|z(t)\|^2+|z(0,t)|^2 \leq c_{\mathcal L}\Big(\|\theta(t)\|^2+\|\rho(t)\|^2+|X(t)|^2\Big).\label{eq:inverse_bound_stochastic}
\end{align}

\begin{remark}
\em{The perturbation term $\Gamma(t)$ in \eqref{eq:target2-5} originates from the use of neural-operator–approximated kernels $(\hat K_0,\hat N_0,\hat\gamma)$ in the controller \eqref{eq:U_NO}, instead of the exact backstepping kernels $(K_0,N_0,\gamma)$. This term acts as a boundary disturbance entering the target system and captures the cumulative effect of kernel approximation errors. In the subsequent analysis, $\Gamma(t)$ will be explicitly bounded in terms of the approximation accuracy of the neural operator.}
\end{remark}
	
Moreover, we anticipate that all the terms on the right-hand side of equation \eqref{eq:target2-4} become small if the stochastic parameters are close enough to the nominal ones. This will be made formal in Lemma \ref{lemma} in Section \ref{subsec:Lyapunov analysis}. More precisely, we have the following lemma.
\begin{lemma}\label{lemma4.1}
There exists a constant \( M_0 \), such that for any realization \( S(t) \in \mathscr{S} \), for any \( (x,\xi) \in \mathcal{T} \),
\[
	| f_{ij} | < M_0 \sum_{S\in \mathscr{S}} | S(t) - S_0|, \quad i \in \{1, 2, 3, 4, 5\}. 
\]
\end{lemma}

\begin{proof}
Each coefficient $f_{ij}$ is obtained by inserting the nominal kernels $(K_0,N_0,\gamma)$ into the transformed dynamics associated with the stochastic mode $\alpha_j$ and subtracting the nominal target relations. Therefore, each $f_{ij}$ is an affine combination of parameter differences of the form
\[
\Lambda_j^\pm-\Lambda_0^\pm,\quad \Sigma_j^{\pm\pm}-\Sigma_0^{\pm\pm},\quad Q_j-Q_0,
\]
with coefficients depending only on the bounded nominal kernels and their first derivatives on $\mathcal T$. Since $\Omega$ is compact and Lemma III.1 guarantees that $K_0$, $N_0$, and $\gamma$ are bounded uniformly with respect to $\alpha_0\in\Omega$, there exists a constant $M_0>0$ such that
\[
|f_{ij}|<M_0\sum_{S\in\mathcal S}|S(t)-S_0|,\qquad i\in\{1,2,3,4,5\}.
\]
\end{proof}

\subsection{Lyapunov analysis}\label{subsec:Lyapunov analysis}
We conduct here a Lyapunov analysis for the stochastic system introduced in Section \ref{sec:target_stochastic}. Consider the following stochastic Lyapunov functionals
\begin{align}\label{eq:lyapunov_stochastic}
    V(t) = \int_0^1(\mathcal{L}_0\mathbf{w})^TD(x,t)\mathcal{L}_0\mathbf{w}dx +X(t)^TPX(t)	
\end{align}
and 
\begin{align*}
    V_j = \int_0^1(\mathcal{L}_0\mathbf{w})^TD_j(x)\mathcal{L}_0\mathbf{w}dx +X(t)^TPX(t),
\end{align*}
where $D(x,t)=D_j(x)$ if $\alpha(t)=\alpha_j$ and
\begin{equation*}
	D_j(x)=\operatorname{Diag}\left\{\frac{e^{-\frac{\nu}{\lambda_1^j} x}}{\lambda_1^j}, \frac{e^{-\frac{\nu}{\lambda_2^j} x}}{\lambda_2^j}, \frac{e^{-\frac{\nu}{\lambda_3^j} x}}{\lambda_3^j}, \frac{ae^{\frac{\nu}{\Lambda_j^-}x}}{\Lambda_j^-}\right\}.
\end{equation*}
If we denote $L_j$ the infinitesimal generator of $V$ obtained by fixing $\alpha(t)=\alpha_j\in \mathscr{S}$, we have
\begin{align}
    L_{j} V(\Psi)=&\frac{d V \left(\Psi, \alpha_{j}\right)}{d \Psi} h_{j}(\Psi)+\sum_{\ell \in \mathfrak{R}}\left(V_{\ell}(\Psi)-V_{j}(\Psi)\right) \tau_{j \ell}(t),\label{eq:LjV}
\end{align}
where $V_{\ell}(\Psi)=V_{\ell}(\Psi,\alpha_\ell)$ and the operator $h_j(\Psi)$ is given by
\begin{align*} 
    h_j(\Psi)=& \begin{pmatrix}
        (A+BF)X(t)+B{\rho}(0,t) 
        \\[10pt]
        -\Lambda_j^+ \theta_x(x, t) + \Sigma_j^{++} \theta(x, t) + \Sigma_j^{+-} \rho(x, t) +\int_0^x C_j^+(x, \xi) \theta(\xi, t)\, d \xi 
        \\
        + \int_0^x C_j^-(x, \xi) \rho(\xi, t)\, d \xi + D_j(x)X(t) 
        \\[10pt]
        \Lambda_j^- \rho_x(x, t) + \Sigma_j^{--} \rho(x,t) + f_{1j}(x)w(x,t) + f_{2j}(x) z(0,t)+f_{3j}(x)X(t)
        \\
        +\int_0^xf_{4j}(x,\xi)w(\xi,t)\,d \xi + \int_0^xf_{5j}(x,\xi)z(\xi,t)\,d \xi
    \end{pmatrix}
\end{align*}

Moreover, since $\lambda_i^j,i=1,2,3$ and $\Lambda_j^-$ in $D_j(x)$ are bounded, the functional $V$ is equivalent to the $L^2$-norm of $(\theta,\rho,X)$ and, consequently, to the $L^2$-norm of the original state $(w,z,X)$ due to \eqref{backstepping_equivalence}. In particular, there exist two positive constants $k_1, k_2 > 0$ such that 
\begin{align}\label{V_equivalence}
    k_1 \Big(\|\theta(t)\| + \|\rho(t)\| + |X|\Big)^2 \leq V(t) \leq k_2 \Big(\|\theta(t)\| + \|\rho(t)\|+ |X|\Big)^2. 
\end{align}

\medskip 
\noindent We have the following lemma.

\begin{lemma}\label{lemma}
There exist constants $\bar{\eta}>0$, $M_1>0$, $M_2>0$ and $d_2>0$ such that the Lyapunov functional $V(t)$ satisfies
\begin{align}
    \sum_{j=1}^r P_{i j} (0, t) L_j V(t) \leq& \, -V(t)\bigg(\bar{\eta} - M_2\varepsilon^2 - M_{1}\sum_{S\in \mathscr{S}}\mathbb{E}\left(|S(t) - S_{0}|\right) \bigg) \nonumber 
    \\
    &+ \sum_{k=1}^3 \Big( d_2 \sum_{S\in \mathscr{S}}\mathbb{E}\left(|S(t) - S_0|\right) - e^{-\frac{{\nu}}{\bar{\Lambda}}} \Big) \theta_k^2(1, t),
\end{align}
where $\bar{\Lambda}:=\sup_{j\in\mathfrak R}\Lambda_j^-$.
\end{lemma}

\medskip 
\begin{proof}
Recall that
\[
V_j(t)=\int_0^1\Theta(x,t)^TD_j(x)\Theta(x,t)\,dx+X(t)^TPX(t),\qquad \Theta=(\theta,\rho),
\]
with
\[
D_j(x)=\operatorname{Diag}\left(\frac{e^{-\nu x/\lambda_1^j}}{\lambda_1^j},\frac{e^{-\nu x/\lambda_2^j}}{\lambda_2^j},\frac{e^{-\nu x/\lambda_3^j}}{\lambda_3^j},\frac{ae^{\nu x/\Lambda_j^-}}{\Lambda_j^-}\right).
\]
Differentiating $V_j$ along the stochastic target system gives
\begin{align*}
	\frac{dV_j}{dt} =& \;2X^TP\big((A+BF)X+B\rho(0,t)\big) +2\int_0^1\sum_{m=1}^3\frac{e^{-\nu x/\lambda_m^j}}{\lambda_m^j}\theta_m(x,t)\theta_{m,t}(x,t)\,dx 
	\\
	&+2a\int_0^1\frac{e^{\nu x/\Lambda_j^-}}{\Lambda_j^-}\rho(x,t)\rho_t(x,t)\,dx.
\end{align*}
We first treat the ODE part. Since $P(A+BF)+(A+BF)^TP=-Q$,
\[
2X^TP(A+BF)X=-X^TQX.
\]
Hence
\[
2X^TP\big((A+BF)X+B\rho(0,t)\big)\le -\lambda_{\min}(Q)|X|^2+2|PB||X||\rho(0,t)|.
\]
Applying Young's inequality yields
\[
2|PB||X||\rho(0,t)|\le \frac{\lambda_{\min}(Q)}{2}|X|^2+c_3\rho(0,t)^2.
\]
For the transport part, integrating by parts componentwise gives
\begin{align*}
	2\int_0^1\frac{e^{-\nu x/\lambda_m^j}}{\lambda_m^j}\theta_m(-\lambda_m^j\theta_{m,x})\,dx
	=-\left[e^{-\nu x/\lambda_m^j}\theta_m^2(x,t)\right]_{x=0}^{x=1} -\nu\int_0^1\frac{e^{-\nu x/\lambda_m^j}}{\lambda_m^j}\theta_m^2(x,t)\,dx.
\end{align*}
Summing over $m=1,2,3$ yields the dissipative boundary term
\[
-\sum_{m=1}^3e^{-\nu/\lambda_m^j}\theta_m^2(1,t)+\theta(0,t)^T\theta(0,t).
\]
Similarly, for the $\rho$-transport term,
\begin{align*}
	2a\int_0^1\frac{e^{\nu x/\Lambda_j^-}}{\Lambda_j^-}\rho\,\Lambda_j^-\rho_x\,dx
	=&\,a\left[e^{\nu x/\Lambda_j^-}\rho^2(x,t)\right]_{x=0}^{x=1}
	-a\nu\int_0^1\frac{e^{\nu x/\Lambda_j^-}}{\Lambda_j^-}\rho^2(x,t)\,dx
	\\
	=&\,ae^{\nu/\Lambda_j^-}\rho^2(1,t)-a\rho^2(0,t)
	-a\nu\int_0^1\frac{e^{\nu x/\Lambda_j^-}}{\Lambda_j^-}\rho^2(x,t)\,dx.
\end{align*}

The lower-order terms containing $\Sigma_j^{++}$, $\Sigma_j^{+-}$, $C_j^\pm$ and $D_j$ are bounded because the mode set is finite and the kernels are bounded. Therefore there exists $c_4>0$ such that
\[
\left|\int_0^1 \Theta^T\mathcal B_j(x)\Theta\,dx\right|\le c_4V_j(t)
\]
for the block of lower-order terms collected in $\mathcal B_j$. Choosing $\nu$ and $a$ as in the nominal proof, these contributions are absorbed by the negative bulk terms, so that the nominal part satisfies
\[
\frac{dV_j}{dt}\le -\eta_0V_j(t)-\sum_{m=1}^3e^{-\nu/\lambda_m^j}\theta_m^2(1,t)+ae^{\nu/\Lambda_j^-}\rho^2(1,t)+\mathfrak R_j(t),
\]
where $\mathfrak R_j$ collects only the residual mismatch terms involving $f_{ij}$ and the cross term $2X^TPB\rho(0,t)$.

We next estimate the boundary perturbation. By the stochastic target boundary condition,
\[
\rho(1,t)=(R_j-R_0)\theta(0,t)+\Gamma(t).
\]

Because $R_j-R_0$ is one of the stochastic parameter deviations, there exists a positive constant $c_5>0$ such that $|R_j-R_0|\le c_5\delta_j(t)$. Hence
\begin{align*}
	\rho^2(1,t) \le 2|R_j-R_0|^2\,|\theta(0,t)|^2+2\Gamma(t)^2 \le 2c_5^2\delta_j(t)^2|\theta(0,t)|^2+2\Gamma(t)^2.
\end{align*}

Since the admissible parameter set is bounded, there exists $\delta_{\max}>0$ such that $\delta_j(t)\le \delta_{\max}$ for every mode $j$. Using the boundary relation $\theta(0,t)=Q_j\rho(0,t)+C_jX(t)$ and Young's inequality, we get
\[
|\theta(0,t)|^2\le c_{15}\big(\rho(0,t)^2+|X(t)|^2\big)\le c_{16}V(t).
\]
Therefore
\[
2c_5^2\delta_j(t)^2|\theta(0,t)|^2\le 2c_5^2\delta_{\max}c_{16}\,\delta_j(t)V(t).
\]

After enlarging the constant $M_1$ if necessary, this contribution is absorbed into the term $M_1\delta_j(t)V(t)$. Now consider the residual terms. Writing only their structure,
\begin{align*}
	\mathfrak R_j(t) =&\,2a\int_0^1\frac{e^{\nu x/\Lambda_j^-}}{\Lambda_j^-}\rho(x,t)\Big(f_{1j}(x)w(x,t)+f_{2j}(x)z(0,t)+f_{3j}(x)X(t)\Big)\,dx
	\\
	&+2a\int_0^1\frac{e^{\nu x/\Lambda_j^-}}{\Lambda_j^-}\rho(x,t)\int_0^x\Big(f_{4j}(x,\xi)w(\xi,t)+f_{5j}(x,\xi)z(\xi,t)\Big)\,d\xi\,dx +2X^TPB\rho(0,t).
\end{align*}

By Lemma V.1, $|f_{ij}|\le M_0\delta_j(t)$. Using Young's inequality and the previous inverse estimate, each term is bounded by a constant multiple of $\delta_j(t)V(t)$. For instance,
\begin{align*}
	\left|2a\int_0^1\frac{e^{\nu x/\Lambda_j^-}}{\Lambda_j^-}\rho f_{1j}w\,dx\right|
	\le c_6\delta_j(t)\int_0^1\big(\rho^2+w^2\big)\,dx \le c_7\delta_j(t)V(t),
\end{align*}
and likewise
\[
\left|2a\int_0^1\frac{e^{\nu x/\Lambda_j^-}}{\Lambda_j^-}\rho f_{2j}z(0,t)\,dx\right|+\left|2a\int_0^1\frac{e^{\nu x/\Lambda_j^-}}{\Lambda_j^-}\rho f_{3j}X\,dx\right|\le c_8\delta_j(t)V(t),
\]
while the double integrals with $f_{4j}$ and $f_{5j}$ are controlled by Cauchy--Schwarz and the triangular structure of $\mathcal T$:
\[
\left|\int_0^1\rho(x)\int_0^x f_{4j}(x,\xi)w(\xi)\,d\xi\,dx\right|+\left|\int_0^1\rho(x)\int_0^x f_{5j}(x,\xi)z(\xi)\,d\xi\,dx\right|\le c_9\delta_j(t)V(t).
\]
Finally, the neural-operator approximation from Corollary IV.1 implies
\[
|\Gamma(t)|\le \varepsilon\left(\|w(t)\|_{L^1}+\|z(t)\|_{L^1}+|X(t)|\right)\le c_{10}\varepsilon\sqrt{V(t)},
\]
hence
\[
\Gamma(t)^2\le c_{10}^2\varepsilon^2V(t).
\]
Collecting all previous inequalities gives
\[
\frac{dV_j}{d\Psi}h_j(\Psi)\le -\eta_0V_j(t)+M_2\varepsilon^2V(t)+M_1\delta_j(t)V(t)+\sum_{m=1}^3\Big(d_2\delta_j(t)-e^{-\nu/\bar\Lambda^-}\Big)\theta_m^2(1,t),
\]
which proves the lemma.
\end{proof}

\medskip 
\noindent We have now all the ingredients to present the proof of Theorem \ref{them_Main2}.

\begin{proof}[Proof of Theorem \ref{them_Main2}]
Choose $\varepsilon^*>0$ such that $M_2(\varepsilon^*)^2\le \bar\eta/4$. Next choose $\delta^*>0$ so small that
\[
	d_2\delta^*<e^{-\frac{\nu}{\bar\Lambda^-}} \quad\text{and}\quad M_1\delta^*<\bar\eta/4.
\]
If
\[
	\sum_{S\in\mathscr S}\mathbb E_{[0,S(0)]}\big(|S(t)-S_0|\big)\le \delta^*,\qquad t\geq 0,
\]
then Lemma \ref{lemma} yields
\[
	\sum_{j=1}^rP_{ij}(0,t)L_jV(t)\le -\frac{\bar\eta}{2}V(t).
\]
Applying Dynkin's formula to $V$ gives
\[
	\mathbb E(V(t))\le e^{-\frac{\bar\eta}{2}t}V(0).
\]
Finally, by the norm equivalence \eqref{V_equivalence} and \eqref{backstepping_equivalence}, there exist $\varsigma,\zeta>0$ such that
\[
	\mathbb{E}_{(0,(p(0),S(0)))}[p(t)]\le \varsigma e^{-\zeta t}p(0),
\]
which proves the mean-square exponential stability of the closed-loop system.
\end{proof}

\section{Numerical simulations}\label{sec:simulations}

\subsection{Simulations' configuration} 
The ODE dynamics are defined by 
\begin{displaymath}
	\begin{array}{llll}
		A = \begin{bmatrix} 0 & 2 \\ -2 & 0 \end{bmatrix}, & B = \begin{bmatrix} 2 \\ 1 \end{bmatrix}, & C = \begin{bmatrix} 1 & 0 \\ 0 & 0 \\ 0 & 0 \end{bmatrix}, & F = \begin{bmatrix} -2 \\ -1 \end{bmatrix}^\top.				
	\end{array}
\end{displaymath}
The nominal parameters are deterministic with values
\begin{displaymath}
	\begin{array}{ll}
		\Lambda_0^+ = \mathrm{diag}(1, 1.01, 0.98), & \Sigma_0^{++} = \mathrm{diag}(0.3, 0.3, 0.3), 
		\\
		\Sigma_0^{+-} = \begin{bmatrix} 0.5 & -0.1 & 0.2 \end{bmatrix}^\top, & \Sigma_0^{-+} = \begin{bmatrix} 0.3 & -0.2 & 0.1 \end{bmatrix},
		\\
		\Sigma_0^{--} = 0.3, & Q_0 = \begin{bmatrix} 1 & 1.05 & 1 \end{bmatrix}^\top, 
		\\
		R_0 = \begin{bmatrix} 1 & 1 & 1 \end{bmatrix}, & \Lambda_0^- = 1.
	\end{array}
\end{displaymath}

For the numerical tests, only the eight scalar quantities $\Lambda^+_{11}$, $\Lambda^-$, $\Sigma^{++}_{11}$, $\Sigma^{+-}_{1}$, $\Sigma^{-+}_{1}$, $\Sigma^{--}$, $Q_1$, and $R_1$ are allowed to vary. All unlisted entries of the matrix- and vector-valued coefficients are held fixed at their nominal values. The six possible values of each selected scalar quantity are reported in Table \ref{table1}.
	
\begin{table}[htbp]
	\centering		
	\begin{tabular}{lcccccc}
		\hline
		Parameter & Mode 1 & Mode 2 & Mode 3 & Mode 4 & Mode 5 & Mode 6 \\
		\hline
		$\Lambda^+_{11}$ & 0.85 & 0.95 & 1.00 & 1.05 & 1.15 & 1.25 \\
		\hline
		$\Lambda^-$ & 0.80 & 0.95 & 1.00 & 1.10 & 1.25 & 1.50 \\
		\hline
		$\Sigma^{++}_{11}$ & 0.15 & 0.25 & 0.30 & 0.35 & 0.45 & 0.55 \\
		\hline
		$\Sigma^{+-}_{1}$ & 0.30 & 0.40 & 0.50 & 0.60 & 0.70 & 0.80 \\
		\hline
		$\Sigma^{-+}_{1}$ & 0.15 & 0.25 & 0.30 & 0.35 & 0.45 & 0.55 \\
		\hline
		$\Sigma^{--}$ & 0.15 & 0.25 & 0.30 & 0.35 & 0.45 & 0.55 \\
		\hline
		$Q_{1}$ & 0.70 & 0.85 & 1.00 & 1.15 & 1.30 & 1.45 \\
		\hline
		$R_{1}$ & 0.70 & 0.85 & 1.00 & 1.15 & 1.30 & 1.45 \\
		\hline
	\end{tabular}
	\vspace{0.2cm}
	\caption{Possible values for each selected stochastic scalar quantity}\label{table1}
\end{table}
	
	The transition probabilities of Markov process are computed by solving the Kolmogorov equation \eqref{eq:kolmogorov} with transition rates $\tau_{ij}(t)$ as in \cite{auriol2023mean}. 

\subsection{Dataset generation and NO training}

To generate the dataset over which we trained DeepONet, we sampled $5000$ parameter vectors using the following ranges for these same eight scalar quantities:
\begin{displaymath}
	\begin{array}{llll}
		\Lambda^-\sim U(0.8, 1.5), & \Lambda^+_{11} \sim U(0.8, 1.3), & \Sigma^{++}_{11} \sim U(0.1, 0.6), & \Sigma^{+-}_{1} \sim U(0.2, 1.9),
		\\
		\Sigma^{-+}_{1} \sim U(0.1, 0.6), & \Sigma^{--} \sim U(0.1, 0.6), & Q_1 \sim U(0.6, 1.5), & R_1 \sim U(0.6, 1.5),		
	\end{array} 
\end{displaymath}

All remaining entries are fixed at their nominal values. With $U$ denoting the uniform distribution over the interval, we solved the kernel equations \eqref{eq:kernel_1} for each sampled parameter vector. In this way, we obtained a dataset of $5000$ tuples $(K_0, N_0, \gamma)$. With this dataset, we trained DeepONet using an Nvidia RTX 4060 Ti GPU for $600$ epochs. 

With the trained DeepOnet model, we have computed new backstepping kernels (not included in the training dataset) and compared them with their numerically computed counterparts. In Table \ref{ta:kernel}, we see that this comparison is quite accurate. Moreover, Table \ref{tab:nopspeedups} displays the computational times of the two approaches under different levels of stochasticity. As the spatial resolution increases, the neural operator maintains a nearly constant runtime, achieving speedups of up to several thousand times. 

\begin{table}[h]
	\centering
	\begin{tabular}{c|c|c}
		\hline
		Kernel & Mean absolute error & Max absolute error 
		\\ 
		\hline $K_{0}$ & $3.62 \times 10^{-4}$  & $2.32\times 10^{-2}$ 
		\\
		$N_{0}$ & $7.42 \times 10^{-4}$  & $1.97\times 10^{-2}$ 
		\\
		$\gamma$ & $6.47 \times 10^{-4}$  & $1.56\times 10^{-2}$ 
		\\		
		\hline
	\end{tabular}
	\vspace{0.2cm}
	\caption{Errors between numerically computed and DeepONet kernels}\label{ta:kernel}
\end{table}

\subsection{Backstepping stabilization via NO}

We show through an illustrative example that the NO-computed kernels yield an effective control for the original PDE-ODE model \eqref{eq:5}. To this end, we set the time interval $t\in (0,70)$, the initial data $w(x,0) = \sin(2\pi x)\,[1\; 1\; 1]^T,z(x,0) = x, X(0) = [1\; -1]^T,$ and apply the NO-computed control to the corresponding dynamics. Figure \ref{closed_loop} shows that this control is capable of stabilizing the dynamics. Finally, we assessed the learned operator for out-of-distribution values of the parameter $\Lambda^-$. For each value, we computed the sum of the norms of all system states at the final time under the corresponding NO control. The results in Table \ref{tab:ood_lambda} show that moderately out-of-distribution parameter values still yield accurate predictions, whereas performance deteriorates and becomes unreliable for values far from the training distribution. This highlights the importance of adequately covering the parameter space during training to ensure robust generalization.
\begin{table}[htbp]
	\centering	
	\begin{tabular}{lccc}
		\hline
		\multicolumn{1}{c}{\begin{tabular}[c]{@{}l@{}}Spatial step\\ Size (dx)\end{tabular}} & 
		\multicolumn{1}{c}{\begin{tabular}[c]{@{}c@{}}Numerical kernel\\ computational time\end{tabular}} & 
		\multicolumn{1}{c}{\begin{tabular}[c]{@{}c@{}}NO Kernel\\ computational time \end{tabular}} & 
		\multicolumn{1}{c}{Speedup $\uparrow$} \\
		\hline
        \multicolumn{4}{c}{\textbf{1 random parameters($\Lambda^-$)}} \\
		0.01 & 3.104 & 0.023 & 135$\times$ \\
        0.005 & 11.805  & 0.025   & 472$\times$\\
		0.001 & 261.50 & 0.036 & 7264$\times$ \\
		\hline
		\multicolumn{4}{c}{\textbf{All eight selected scalar parameters}} \\
		0.01 & 3.456 & 0.025 & 138$\times$ \\
        0.005 & 12.805  & 0.035   & 365$\times$\\
		0.001 & 290.10 & 0.040 & 7253$\times$ \\
		\hline
	\end{tabular}\label{tab:nopspeedups} 
	\vspace{0.2cm}
	\caption{Computational performance}
	\label{tab:performance_comparison}
\end{table}
    
\begin{table}[htbp]
		\centering		
		\begin{tabular}{c c c}
			\hline
			$\Lambda^{-}$ value & Distribution type & Final-state norm\\
			\hline
			0.80 & In-distribution & $1.21 \times 10^{-5}$ (Stable) \\
			1.10 & In-distribution & $ 1.44\times 10^{-5}$ (Stable) \\
			2.00 & Out-of-distribution & $ 1.67\times 10^{-3}$ (Stable) \\
			2.50 & Out-of-distribution & $5 \times 10^{3}$ (Unstable)  \\
			\hline
		\end{tabular}
		\vspace{0.2cm}
		\caption{Generalization performance under in-distribution and out-of-distribution values of $\Lambda^{-}$}
		\label{tab:ood_lambda}
	\end{table}
    
\begin{figure}[htbp!]
    \centering
    \includegraphics[width=0.75\textwidth]{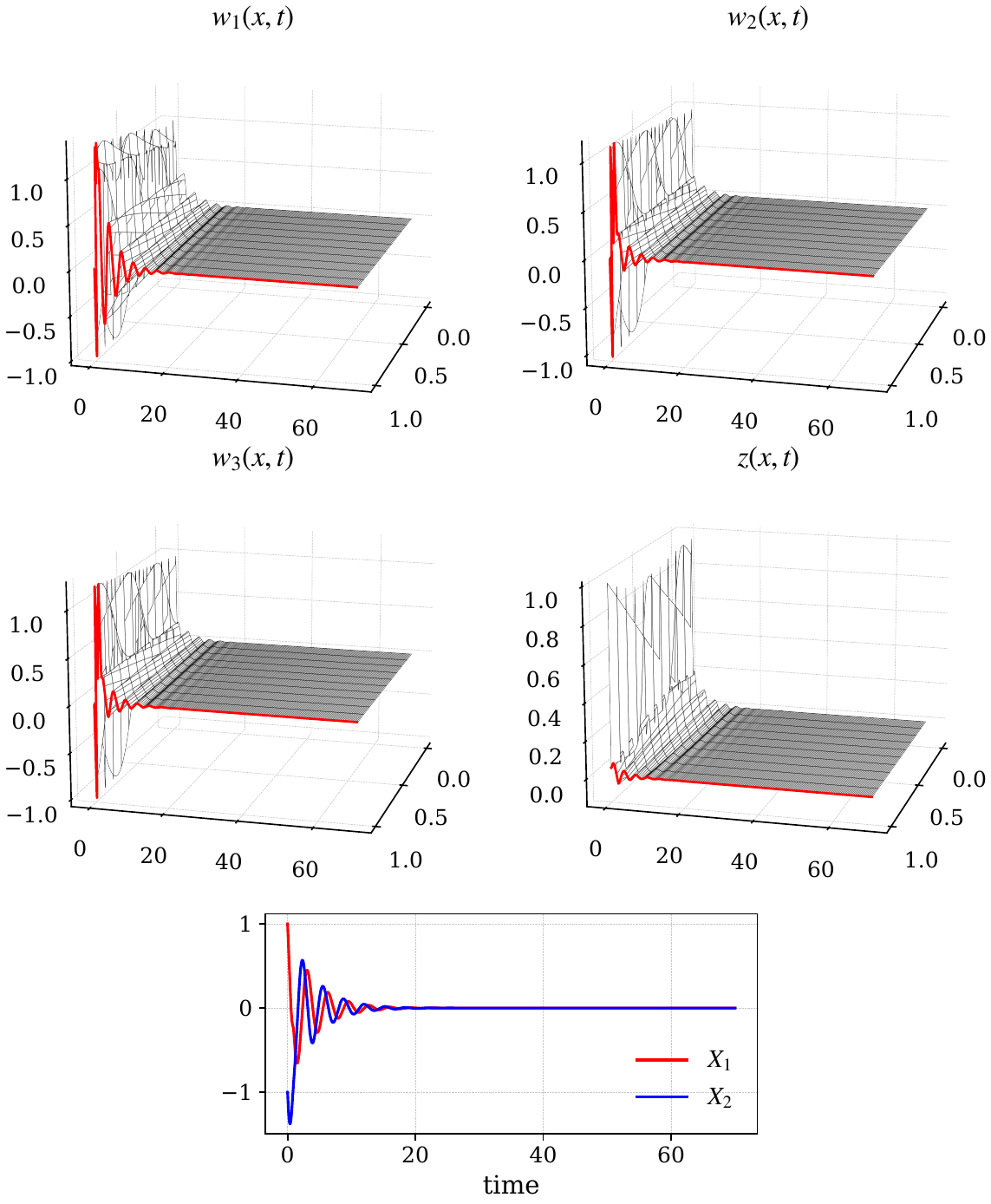}
    \caption{The states of closed-loop system with NO-based controller.}\label{closed_loop}
\end{figure} 

\section{Conclusions and open problems}\label{sec:conclusions}
In this paper, we proposed a NO-based framework for the mean-square exponential stabilization of coupled PDE–ODE systems with Markovian switching parameters.  Numerical results show over two orders of magnitude speedup while preserving accuracy and closed-loop stability under random switching. Future work will address extensions to multi-dimensional PDEs and stochastic ODE matrices.

\section*{Acknowledgments} We thank Prof. Enrique Zuazua (Friedrich-Alexander-Universit\"at Erlangen-N\"urnberg, University of Deusto and Universidad Aut\'onoma de Madrid) and Prof. Francisco Periago (Technical University of Cartagena, Spain) for fruitful and interesting discussions over the topics of this paper.  

\appendix{
\section{Derivations of gain kernel PDEs}\label{app:kernel}

The kernel equations are obtained by differentiating the transformation \eqref{eq:transformation} with respect to $t$ and $x$, substituting the nominal plant dynamics \eqref{eq:system_nomimal}, and integrating the Volterra terms by parts. After this substitution, the residual
\[
\rho_t(x,t)-\Lambda_0^-\rho_x(x,t)-\Sigma_0^{--}\rho(x,t)
\]
can be written as the sum of terms proportional to $w_{nom}(x,t)$, $z_{nom}(0,t)$, $X(t)$, and Volterra integrals in $w_{nom}$ and $z_{nom}$. Since these quantities are independent, each corresponding coefficient must vanish. Matching the coefficients gives precisely the kernel equations (III.4):
\begin{align*}
	& \Lambda_0^-(K_0)_x(x,\xi)=(K_0)_\xi(x,\xi)\Lambda_0^++N_0(x,\xi)\Sigma_0^{-+}+K_0(x,\xi)(\Sigma_0^{++}-\Sigma_0^{--}I_3)
	\\
	& \Lambda_0^-(N_0)_x(x,\xi)+\Lambda_0^-(N_0)_\xi(x,\xi)=K_0(x,\xi)\Sigma_0^{+-}
	\\
	& K_0(x,x)\big(\Lambda_0^-I_3+\Lambda_0^+\big)=-\Sigma_0^{-+}
	\\
	& \Lambda_0^-N_0(x,0)-K_0(x,0)\Lambda_0^+Q_0=\gamma(x)B
	\\
	& \Lambda_0^-\gamma'(x)=\gamma(x)A-\Sigma_0^{--}\gamma(x)-K_0(x,0)\Lambda_0^+C,
	\\
	& \gamma(0)=F,
\end{align*}
where $I_3$ is the $3\times3$ identity matrix. The condition $\gamma(0)=F$ is determined by the boundary matching at $x=0$: indeed,
\[
\rho(0,t)=z_{nom}(0,t)-FX(t),
\]
and therefore
\[
\theta(0,t)=w_{nom}(0,t)=Q_0\rho(0,t)+(C+Q_0F)X(t),
\]
which is exactly the left boundary condition of the target system.
}
\bibliographystyle{acm}
\bibliography{refDataBase} 
	
\end{document}